\documentclass[12pt]{amsart}

\usepackage[latin1]{inputenc}
\usepackage[english]{babel}
\usepackage{amsthm}
\usepackage{amsmath}
\usepackage{amsxtra}
\usepackage{mathrsfs}
\usepackage{amssymb}
\usepackage[dvips]{graphicx}
\usepackage{graphicx}

\input xy
\xyoption{all}

\newtheorem{teo}{Theorem}[section]
\newtheorem{lemma}[teo]{Lemma}

\newtheorem{defi}[teo]{Definition}

\newtheorem{fatto}[teo]{Fact}

\newtheorem{cla}[teo]{Claim}
\newtheorem{prop}[teo]{Proposition}

\theoremstyle{remark}
\newtheorem{remark}[teo]{Remark}

\begin{document}

\newcommand{\ran}{\texttt{ran}}
\newcommand{\cof}{\texttt{cof}}
\newcommand{\dom}{\texttt{dom}}
\newcommand{\N}{\mathcal{N}}
\newcommand{\up}{\upharpoonright}
\newcommand{\urltilde}{\kern -.15em\lower .7ex\hbox{~}\kern .04em}

\newcommand\Acal{\mathscr{A}}
\newcommand\Bcal{\mathscr{B}}
\newcommand\Ecal{\mathscr{E}}
\newcommand\Fcal{\mathscr{F}}
\newcommand\Hcal{\mathscr{H}}
\newcommand\Ical{\mathscr{I}}
\newcommand\Ncal{\mathscr{N}}
\newcommand\Mcal{\mathscr{M}}
\newcommand\Pcal{\mathscr{P}}
\newcommand\Qcal{\mathscr{Q}}
\newcommand\Rcal{\mathscr{R}}
\newcommand\Tcal{\mathscr{T}}
\newcommand\Ucal{\mathscr{U}}
\newcommand\Zcal{\mathscr{Z}}
\newcommand\Rbb{\mathbb{R}}
\newcommand\Nfrak{\mathfrak{N}}
\newcommand\Pfrak{\mathfrak{P}}
\newcommand\restrict{\restriction}

\newcommand{\diff}{\operatorname{\mathrm diff}}
\newcommand{\Ht}{\operatorname{\mathrm ht}}
\newcommand{\lev}{\operatorname{\mathrm lev}}
\newcommand{\meet}{\wedge}
\newcommand\triord{\triangleleft}
\newcommand{\Th}{{}^{\mathrm{th}}}
\newcommand{\St}{{}^{\mathrm{st}}}
\newcommand\axiom{\mathrm}
\newcommand\MA{\axiom{MA}}
\newcommand\MM{\axiom{MM}}
\newcommand\PFA{\axiom{PFA}}
\newcommand\BPFA{\axiom{BPFA}}
\newcommand\MRP{\axiom{MRP}}
\newcommand\SRP{\axiom{SRP}}
\newcommand\ZFC{\axiom{ZFC}}
\newcommand\CAT{\axiom{CAT}}
\newcommand\CH{\axiom{CH}}
\newcommand{\<}{\langle}
\renewcommand{\>}{\rangle}
\newcommand\mand{\textrm{ and }}
\renewcommand{\diamond}{\diamondsuit}
\newcommand{\forces}{\Vdash}

\newcommand{\cf}{{\mbox{cof}}}
\newcommand{\height}{\Ht}
\newcommand\NS{\mathrm{NS}}

\title[Five element basis]{A direct proof of the five element basis theorem}

\author[Boban Veli\v{c}kovi\'{c}]{Boban Veli\v{c}kovi\'{c}}
\email{boban@math.univ-paris-diderot.fr}
\urladdr{http://www.logique.jussieu.fr/\urltilde boban}
\author{Giorgio Venturi}
\email{gio.venturi@gmail.com}
\address{IMJ-PRG,
Universit\'e Paris Diderot,
 75205 Paris Cedex 13, France}



\begin{abstract}{We present a direct proof of the consistency
of the existence of a five element basis for the uncountable
linear orders. Our argument is based on the approach
of Larson, Koenig, Moore and Velickovic
and simplifies the original
proof of Moore.}

\end{abstract}
\maketitle

\section*{Introduction}

In \cite{linear_basis} Moore showed that $\PFA$ implies that the class
of the uncountable linear orders has a five element basis,
i.e., that there is a list of five uncountable
linear orders such that every uncountable linear order contains an
isomorphic copy of one of them. This basis consists of $X$,
$\omega_1$, $\omega_1^*$, $C$,  and $C^*$, where $X$ is any suborder of
the reals of cardinality $\omega_1$ and $C$ is any
Countryman line\footnote{Recall that a {\em Countryman line} is an uncountable
linear order whose square is the union of countably many non-decreasing relations.
The existence of such a linear order was proved by Shelah in \cite{Countryman:Shelah}.}.
It was previously known from the work of Baumgartner \cite{aleph_1-dense}
and Abraham-Shelah \cite{club_isomorphic}, that, assuming a rather
weak forcing axiom, the existence of a five element linear basis
for uncountable linear orderings is equivalent to the following
statement, called the {\em Coloring Axiom for Trees} ($\CAT$):
\begin{quote}
There is a normal Aronszajn tree $T$ such that for every $K\subseteq T$
there is an uncountable antichain $X\subseteq T$ such that
$\wedge(X)$ is either contained in or disjoint from $K$.
\end{quote}
Here, $\wedge(X)$ is the set of all pairwise meets of
incomparable elements of $X$.

\noindent One feature of the argument from \cite{linear_basis}
is that it relies crucially on the Mapping
Reflection Principle ($\MRP$), a strong combinatorial principle
previously introduced by Moore in \cite{MRP},
in order to prove the properness of the appropriate forcing notion.
It was shown in \cite{MRP} that $\MRP$ implies the failure of $\square_\kappa$,
for all $\kappa \geq \omega_1$, and therefore its consistency
requires very large cardinal axioms.
However, it was not clear if any large cardinals  were needed
for the relative consistency of $\CAT$.
Progress on this question was made by
K\"onig, Larson, Moore and Veli\v{c}kovi\'{c} in \cite{linear_phi}
who  reduced considerably the large cardinal assumptions in Moore's proof.
They considered a statement $\varphi$ which is a form of saturation
of Aronszajn trees and showed that it can be used instead of $\MRP$
in the proof of the Key Lemma (Lemma 5.29) from \cite{linear_basis}.
Moreover, they showed that for the consistency of BPFA together with $\varphi$
it is sufficient to assume the existence of a reflecting Mahlo cardinal.
If one is only interested in the consistency of the existence of a  five element
basis for the uncountable linear orderings then even an smaller large
cardinal assumption is sufficient (see \cite{linear_phi} for details).

The purpose of this note is to present a direct proof of $\CAT$,
and therefore the existence of a five element linear basis,
assuming the conjunction of $\BPFA$ and $\varphi$. The argument
is much simpler than the original proof from \cite{linear_basis}.
It is our hope that by further understanding this forcing one
will be able to determine if any large cardinal assumptions
are needed for the consistency of $\CAT$.

The paper is organized as follows. In \S 1 we present the background
material on Aronszajn trees and the combinatorial principles $\psi$ and $\varphi$.
In \S 2 we start with a coherent special Aronszajn tree $T$ and a subset
$K$ of $T$, define a new coloring of finite subsets of $T$ and prove some
technical lemmas.  In \S 3 we define the main forcing notion $\partial^*(K)$
and show that it is proper. In \S 4 we complete the proof that $\BPFA$ together
with $\varphi$ implies $\CAT$.

\section{Saturation of Aronszajn trees}

Recall that a {\em tree}  is a partially ordered $(T,<)$ such that
for every $t\in T$ the set of predecessors of $t$, i.e.
 $\{ s\in T: s <t\}$, is well ordered. If $t\in T$ we
 let $\Ht(t)$ denote the height of $t$ in $T$, i.e. the order type
 of $\{ s\in T: s <t\}$.  If $t\in T$ and $\xi < \Ht(t)$ we let $t\restriction \xi$
 denote the unique predecessor of $t$ of height $\xi$.
 For an ordinal $\alpha$ we let $T_\alpha$ denote the $\alpha$-th level of $T$
 i.e. the set of all $t\in T$ of height $\alpha$. 
 $T$ is called {\em normal} if it has a unique least element, i.e. the {\em root}, 
 any two nodes of limit height that have the same predecessors are actually equal. 
 If $T$ is normal and $s,t\in T$ then there is the largest node, denoted
 by $s\wedge t$, below $s$ and $t$. We also refer to $s\wedge t$ as the {\em meet}
 of $s$ and $t$. 
 The {\em height} of $T$, denoted by $\Ht(T)$, is
 the least $\alpha$ such that $T_\alpha$ is empty. A {\em chain} in $T$
 is a totally ordered subset $C$ of $T$. An {\em antichain} in $T$ is a subset $A$ of
 $T$ such that any two elements of $A$ are pairwise incomparable. 
 By an \emph{Aronszajn tree} or simply an \emph{A-tree}
we mean a tree of height $\omega_1$ in which all levels and chains are at most countable.
A \emph{subtree} of an A-tree $T$ is an uncountable
downwards closed subset of $T$. 
 We start by discussing the notion of saturation of an Aronszajn tree.

\begin{defi}
An Aronszajn tree $T$ is \emph{saturated}
if, whenever $\Acal$ is a collection of subtrees of $T$ such that
the intersection of any two trees in $\Acal$ is at most countable, 
$\Acal$ has cardinality at most $\omega_1$.
\end{defi}
It was shown by Baumgartner in \cite{bases_Atrees} that the following holds
after Levy collapsing an inaccessible cardinal to $\omega_2$ with countable conditions.
\begin{quote}
For every Aronszajn tree $T$, there is a collection $\Bcal$ of
subtrees of $T$ such that $\Bcal$ has cardinality $\omega_1$
and every subtree of $T$ contains an element of $\Bcal$.
\end{quote}
Clearly, this statement implies that every A-tree is saturated. 
In order to obtain $\CAT$ we will use a form of saturation
of A-tree together with $\BPFA$.
In Baumgartner's model $\CH$ holds hence it is not suitable for our purpose.  
It is for this reason that a different approach was
taken in \cite{linear_phi}. We now recall the relevant
definitions from this paper.

If $\Fcal$ is a collection of subtrees of $T$, then $\Fcal^\perp$ is
the collection of all subtrees $B$ of $T$ such that for every $A$ in
$\Fcal$, $A \cap B$ is countable.
If $\Fcal^\perp$ is empty, then $\Fcal$ is said to be \emph{predense}.
For $\Fcal$ a collection of subtrees of an Aronszajn tree $T$, we
consider the following statements:
\begin{description}

\item[$\psi_0(\Fcal)$]
There is a closed unbounded set $E \subseteq \omega_1$ and a
continuous chain $(N_\nu:\nu \in E)$ of countable
subsets of $\Fcal$ such that for every $\nu$ in $E$
and $t$ in $T_\nu$
there is a $\nu_t < \nu$ such that if $\xi \in (\nu_t,\nu)
\cap E$, then there is $A \in N_\xi$ such that
$t \restriction \xi$ is in $A$.

\end{description}
\begin{description}

\item[$\varphi_0(\Fcal)$]
There is a closed unbounded set $E \subseteq \omega_1$ and a
continuous chain $(N_\nu:\nu \in E)$ of countable
subsets of $\Fcal \cup \Fcal^\perp$ such that for every $\nu$ in $E$
and $t$ in $T_\nu$ either

\begin{enumerate}

\item there is a $\nu_t < \nu$ such that if $\xi \in (\nu_t,\nu)
\cap E$, then there is $A \in {\Fcal} \cap N_\xi$ such that $t
\restriction \xi$ is in $A$, or

\item there is a $B$ in ${\Fcal}^\perp \cap N_\nu$ such that $t$
is in $B$.

\end{enumerate}
\end{description}
It is not difficult to show that $\psi_0(\Fcal)$ implies that
$\Fcal$ is predense, indeed that $\bigcup_\nu N_\nu$ is predense.
Hence, if $\psi_0(\Fcal)$ holds for every predense family of subtrees of $T$,
then $T$ is saturated. 
It is also clear that $\psi_0(\Fcal)$ is a $\Sigma_1$-formula
with parameters $\Fcal$ and $T$.
While $\varphi_0(\Fcal)$ and $\psi_0(\Fcal)$ are equivalent if
$\Fcal$ is predense, $\varphi_0(\Fcal)$ is in general
not a $\Sigma_1$-formula in $\Fcal$ and $T$.
Let $\varphi$ be the assertion that if $T$ is an Aronszajn tree
and $\Fcal$ is a family of subtrees $T$ then $\varphi_0(\Fcal)$ holds, and
let $\psi$ be the analogous assertion but with quantification only
over $\Fcal$ that are predense.
As noted, $\varphi$ implies $\psi$.
The following was proved as Corollary 3.9 in  \cite{linear_phi}.

\begin{prop}\label{force_phi(F)}
For a given family $\Fcal$ of subtrees of an Aronszajn tree $T$,
there is a proper forcing extension which satisfies $\varphi_0(\Fcal)$.
\qed
\end{prop}

\begin{remark}
If we want to force $\varphi$ it is natural to start with
an inaccessible cardinal $\kappa$ and do a countable support
iteration of proper forcing notions
$({\mathcal P}_{\alpha},\dot{\mathcal Q}_\beta; \alpha \leq \kappa, \beta < \kappa)$.
At stage $\alpha$ we can use $\diamond_\kappa$ to guess
an Aronszajn tree $\dot{T}_\alpha$ and a family $\dot{{\mathscr F}}_\alpha$
of subtrees of $\dot{T}_\alpha$ in the model $V^{{\mathcal P}_{\alpha}}$
and let $\dot{{\mathcal Q}_{\alpha}}$ be a ${\mathcal P}_\alpha$-name
for the proper poset which forces $\varphi_0(\dot{\mathscr F}_\alpha)$.
Suppose in the final model  $V^{\mathcal P_\kappa}$ we have an Aronszajn tree $\dot{T}$
and a family $\dot{\mathscr F}\in V^{{\mathcal P}_\kappa}$ of subtrees
of $\dot{T}$. In order to ensure that $\varphi_0(\dot{\mathscr F})$ holds
in $V^{\mathcal P_\kappa}$ we need to find a stage $\alpha$ of the iteration
at which $\dot{T}$ and $\dot{\mathscr F}$ are guessed, i.e. $\dot{T}_\alpha=\dot{T}$
and $\dot{\mathscr F}\restriction V^{{\mathcal P}_\alpha}=\dot{\mathscr F}_\alpha$
{\em and} moreover such that
$$
(\dot{\mathscr F}^\perp)^{V^{\mathcal P_\kappa}}\restriction V^{\mathcal P_\alpha} =
(\dot{\mathscr F}^{\perp}_{\alpha})^{V^{{\mathcal P}_\alpha}}.
$$
This is the reason why a Mahlo cardinal is used in the following
theorem from \cite{linear_phi}.
\end{remark}

\begin{teo}
If there is a cardinal which is both reflecting and Mahlo, then there is a
proper forcing extension of $L$ which satisfies the conjunction
of $\BPFA$ and $\varphi$.
In particular the forcing extension satisfies that the uncountable
linear orders have a five element basis.
\qed
\end{teo}

If one is interested only in the consistency of the existence of a five element
basis for the  uncountable linear orderings  it was observed in \cite{linear_phi}
then a somewhat smaller large cardinal is sufficient. Indeed, for the desired
conclusion one does not need the full strength of $\BPFA$ and one
only needs $\varphi_0(\mathscr F)$ for certain families of subtrees
of an Aronszajn tree $T$ which are $\Sigma_1$-definable using
a subset of $\omega_1$ as a parameter. The precise large cardinal assumption
is that there is an inaccessible cardinal $\kappa$ such that
for every $\kappa_0 < \kappa$, there is an inaccessible
cardinal $\delta < \kappa$
such that $\kappa_0$ is in $H(\delta)$ and $H(\delta)$ satisfies
there are two reflecting cardinals greater than
$\kappa_0$.

\section{Colorings of Aronszajn trees}

Let $2^{<\omega_1}$ denote the full binary tree of height $\omega_1$ with the usual ordering.
For the remainder of the paper we fix an Aronszajn
tree $T\subseteq 2^{<\omega_1}$ which is special, coherent, 
and closed under finite modifications.
The tree $T(\varrho_3)$ from  \cite{cseq} is such an example.
Recall that  $T$ is {\em special} if it can be written as $T=\bigcup_n A_n$,
where $A_n$ is an antichain, for all $n$. Notice that this implies
that any uncountable subset of $T$ contains an uncountable antichain. 
We will often use this fact without mentioning it.
Since $T$ is a subtree of $2^{<\omega_1}$,  the $\alpha$-th level of $T$, 
i.e. $T_\alpha$, is simply $T\cap 2^{\alpha}$. 
We say that $T$ is {\em coherent} if, for every countable $\alpha$ and $s,t\in T_\alpha$,
the set
\[
D(s,t)=\{ \xi <\alpha : s(\xi)\neq t(\xi)\}
\]
is finite. Finally, $T$ is {\em closed under finite modifications} if, for every
$\alpha$ and $s,t\in 2^\alpha$, if $s\in T$ and $D(s,t)$ is finite then $t\in T$. 
For $A\subseteq \omega_1$ we set $T\restriction A=\bigcup_{\alpha \in A}T_\alpha$.
If $s$ and $t$ are incomparable nodes in $T$, i.e. if $D(s,t)$ is
non-empty,  we let
$$
\Delta(s,t) =\min D(s,t).
$$
Since $T$ is a subtree of $2^{<\omega_1}$ which is normal it is itself normal. 
If $s,t\in T$ then the meet of $s$ and $t$, i.e. $s\meet t$, is simply
$s\restriction \Delta(s,t)$. Given a subset $X$
of $T$ we let
$$
\wedge(X) = \{ s\wedge t: s,t \in X, s \mbox{ and } t \mbox{ incomparable}\}.
$$
Note that if $T_X$ is the tree induced by $X$, i.e. the set of
all initial segments of elements of $X$, then $\meet(T_X)=\meet(X)$.
We also let
$$
\pi(X)=\{ t\restriction \Ht(s): s,t\in X \mbox{ and } \Ht(s)\leq \Ht(t) \}.
$$
We let $\lev(X)=\{ \Ht(t): t\in X\}$. If $\alpha \in \lev(X)$ we let
$\pi_\alpha(X)=\pi(X)\cap T_\alpha$.

We will also need to consider finite powers of our tree $T$.
Given an integer $n$  and a level $T_\alpha$ of $T$ we let

$$T^{[n]}_\alpha=\{ \tau \in T_\alpha^n: i <j \rightarrow \tau(i)\leq_{\rm lex}\tau(j)\}
$$
where $\leq_{\rm lex}$ denotes the lexicographic ordering of $T$.
We let $T^{[n]}=\bigcup_{\alpha}T^{[n]}_\alpha$.
Morally, elements of $T^{[n]}$ are $n$-element subsets of $T$
of the same height.
In order to ensure that $T^{[n]}$ is closed under taking restrictions,
it is necessary to allow for $n$-element sets with repetitions,
i.e. {\em multisets}, and the above definition is a formal way to accommodate this.
We will abuse notation and identify elements of $T^{[n]}$ that have
distinct coordinates with the set of their coordinates.
In our arguments, only the range of these sequences
will be relevant.

If $\sigma \in T_\alpha^{[n]}$ and $\tau \in T_\alpha^{[m]}$,
for some $\alpha$, then, by abusing notation, we will write
$\sigma \cup \tau$ is the sequence of length $n+m$ which
enumerates the coordinates of $\sigma$ and $\tau$ in
$\leq_{\rm lex}$-increasing order including repetitions.
We will also write $\sigma \subseteq \tau$ if the multiset enumerated by $\sigma$ 
is included in the multiset enumerated by $\tau$ by counting multiplicities. 
$T^{[n]}$ will be considered as a tree with the coordinate-wise
partial order induced by $T$. If $\sigma \in T^{[n]}$ and
$\alpha < \Ht(\sigma)$ we write $\sigma \restriction \alpha$
for the sequence $(\sigma(i)\restriction \alpha : i<n)$.
If $\sigma,\tau \in T^{[n]}$ are incomparable we will let
$$
\Delta(\sigma,\tau) =\min \{ \alpha : \sigma(i)(\alpha)\neq \tau(i)(\alpha),
\mbox{ for some } i <n\}
$$
and we will write $\sigma \meet \tau$ for $\sigma \restriction \Delta(\sigma,\tau)$.

For $\sigma \in T^{[n]}$ let
$$
D_\sigma = ( D(\sigma(i),\sigma(0)): i <n)
$$
Suppose  $\sigma, \tau \in T^{[n]}$ and $\Ht(\sigma)\leq \Ht(\tau)$.
We say that the pair $\{\sigma, \tau\}$ is {\em regular}
if $ D_\tau\restriction \Ht(\sigma) = D_\sigma$, i.e. for all $i<n$,
$$
D(\tau(i),\tau(0))\cap \Ht(\sigma) = D(\sigma(i),\sigma(0)).
$$
Note that in this case, for all $i,j<n$,
$$
\Delta(\tau(i),\sigma(i))=\Delta(\tau(j),\sigma(j)).
$$
We say that a subset $X$ of $T^{[n]}$ is {\em regular}
if every pair of elements of $X$ is regular.
Note that if $X$ is regular, then so is the tree $T_X$
generated by $X$.
A {\em level sequence} of $T^{[n]}$ is a sequence
$\{ \sigma_\alpha: \alpha \in A\}$ where $A$ is
a subset of $\omega_1$ and $\sigma_\alpha \in T^{[n]}_\alpha$, for
all $\alpha \in A$.
The following is a simple application
of the $\Delta$-system lemma and the Pressing Down Lemma.
\begin{fatto}
\label{delta}
Let $\mathcal A=\{ \sigma_\alpha: \alpha \in A\}$
be a level sequence. If $A$ is uncountable (stationary)
then there is an uncountable (stationary) subset $B$
of $A$ such that $\mathcal B= \{ \sigma_\alpha: \alpha \in B\}$
is regular.
\qed
\end{fatto}





From now on we assume the conjunction of $\BPFA$ and $\varphi$.
We are given a subset $K$ of $T$ and we want to find an uncountable
antichain $X$ in $T$ such that $\meet(X)\subseteq K$ or $\meet(X)\cap K=\emptyset$.
We will refer to $K$ as a {\em coloring} of $T$.
We first note that, for every integer $n$, $K$ induces a coloring $K^{[n]}$ of $T^{[n]}$
 defined by
$$
K^{[n]}=K^n\cap T^{[n]}.
$$
We let $\mathscr F_n$ be the collection of regular subtrees $R$ of
$T^{[n]}$ such that $\meet(R)\cap K^{[n]}=\emptyset$.
The following fact is immediate by using Fact~\ref{delta}.

\begin{fatto}\label{uncountableX}
If  $R \in \Fcal_n^{\bot}$ then for every uncountable $X \subseteq R$ there are
incomparable $\sigma, \tau \in X$ such that $\sigma \meet \tau \in K^{[n]}$.

\end{fatto}
\begin{proof}
	We may assume that $X$ is a level sequence, say $X=\{ \sigma_\alpha : \alpha \in A\}$, for some $A\subseteq \omega_1$. 
	Since $T^{[n]}$ is special, by shirking $X$ if necessary, we may assume
	that it is an antichain. By Fact~\ref{delta} we may further assume that $X$ is regular.
	Now, if $\wedge(X)$ were disjoint from $K^{[n]}$, then so would be $\wedge(T_X)$.
	But then $T_X$ would belong to $\Fcal_n$.
	However, $T_X$ is a subtree of $R$ which is orthogonal to all trees in $\Fcal_n$,
	a contradiction.
\end{proof}
By $\varphi_0(\mathscr F_n)$ we can find a club $C_n$ in $\omega_1$
and a continuous increasing chain $(N^n_\xi: \xi \in C_n)$
of countable subsets of $\mathscr F_n \cup (\mathscr F_n)^{\perp}$
witnessing $\varphi_0(\mathscr F_n)$. By replacing each of the $C_n$
by their intersection we may assume that the $C_n$ are all the same
and equal to say $C$. We now define a new coloring of $T^{[n]}\restriction C$
as follows.

\begin{defi}
A node $\sigma \in T^{[n]}\restriction C$ is in $K^{[n]}_\varphi$
if, letting $\alpha$ be the height of $\sigma$, there exists
$R\in (\mathscr F_n)^{\perp} \cap N^n_\alpha$ such that $\sigma \in R$,
i.e. if $\sigma$ is in case (2) of the dichotomy for $\varphi_0(\mathscr F_n)$.
We denote $(T^{[n]}\restriction C)\setminus K^{[n]}_\varphi$ by
$L^{[n]}_\varphi$. We let $K_\varphi =\bigcup_n K^{[n]}_\varphi$ and
$L_\varphi =\bigcup_n L^{[n]}_\varphi$.
\end{defi}

\begin{remark} The induced coloring $T^{[n]}=K^{[n]}_\varphi \cup L^{[n]}_\varphi$,
for $n<\omega$,
is our analog of the notions of {\em acceptance} and {\em rejection}
from \cite{linear_basis}. The main difference is that these notions
are defined in \cite{linear_basis} relative to a given countable elementary submodel
of $H(\omega_2)$ whereas our colorings do not make
reference to any such model. This simplifies considerably
the proof of properness of the main forcing notion we define
in \S 3.
\end{remark}

We now note some useful facts about these induced colorings.

\begin{fatto}\label{countableB}
If there is a node $t$ in $L^{[1]}_\varphi$ whose height is a limit
point of $C$ then there is an uncountable antichain $X$ in $T$ such that
$\meet(X)\cap K=\emptyset$.
\end{fatto}
\begin{proof}
Assume $t$ is such a node and let $\alpha$ be the height of $t$.
By our assumption, case (1) of the dichotomy for $\varphi_0(\mathscr F_1)$
holds for $t$. Therefore, there exists $\eta <\alpha$ such that for
every $\xi \in (\eta, \alpha)\cap C$ there is a $R \in \Fcal_1 \cap N_{\xi}^1$
such that $t\restriction \xi \in R$.
Since $(\eta,\alpha)\cap C$ is non-empty, it follows that $\mathscr F_1$
is non-empty, as well. Now, let $R$ be a member of $\mathscr F_1$.
Identifying $R$ with a  subtree of $T$ we have that  $\wedge (R) \cap K=\emptyset$.
Since $T$ is special, so is $R$ and we can fix an uncountable antichain
$X$ in $R$.  It follows that $\wedge (X)\cap K=\emptyset$,
as required. 
\end{proof}


\begin{fatto}\label{cont}
If $\sigma \in K_\varphi$  and $\alpha=\Ht(\sigma)$
is a limit point of $C$ then there
is $\eta <\alpha$ such that  $\sigma \restriction \xi \in K_\varphi$,
for all  $\xi \in (\eta,\alpha)\cap C$. Similarly for $L_\varphi$.
We refer to this property as {\em continuity} of the induced coloring.
\qed
\end{fatto}

\begin{fatto}\label{gen}
Suppose $S$ is a stationary subset of $C$ and
${\mathcal S} =\{  \sigma_{\xi} : \xi \in S \}$ is a level sequence in $T^{[n]}$
consisting of elements of $K^{[n]}_\varphi$.
Then there exist distinct $\xi,\eta \in S$
such that $\sigma_\xi \meet \sigma_\eta \in K^{[n]}$.
\end{fatto}
\begin{proof}
For $\xi \in S$ since $\sigma_\xi \in K^{[n]}_\varphi$
there exists a tree $R_\xi \in N_\xi^n\cap ({\mathscr F}_n)^{\perp}$
such that $\sigma_\xi \in R_\xi$. Since $(N_\xi^n: \xi \in C)$ is
a continuous increasing sequence of countable sets, we can apply
the Pressing Down Lemma to the function $\xi \mapsto R_\xi$. 
Hence, by shrinking $S$ if necessary,  we may assume that the trees
$R_{\xi}$ are all the same and equal to some $R$.
This simply means that ${\mathcal S} \subseteq R$.
Since $R\in (\mathscr F_n)^{\perp}$, by Fact \ref{uncountableX},
there are $\xi \neq \eta \in S$ such that  $\sigma_{\eta} \land \sigma_{\xi} \in K^{[n]}$,
as required.
\end{proof}

\begin{defi} Let ${\mathcal S} =\{ \sigma_\xi: \xi \in A\}$ be a regular level sequence
in $T^{[n]}$, for some integer $n$. Then ${\mathcal P}_{\mathcal S}$ is the poset consisting of
finite subsets $p$ of ${\mathcal S}$ such that $\meet(p)\cap K^{[n]}=\emptyset$, ordered
by reverse inclusion.
\end{defi}

The following lemma is the main technical result of this section.

\begin{lemma}\label{mainlemma}
Let ${\mathcal S}=\{ \sigma_\alpha: \alpha \in S\}$ and
${\mathcal Z} =\{ \tau_\gamma: \gamma \in Z\}$ be
two regular level sequences in $T^{[n]} $ and $T^{[m]}$
respectively such that $S$ is a stationary subset
of $C$ and ${\mathcal S} \subseteq K^{[n]}_\varphi$. Assume that, for every
$\alpha \in S$ and $\gamma \in Z$, if $\alpha < \gamma$ then
$$
\sigma_\alpha \cup \tau_\gamma \restriction \alpha \in L_\varphi^{[n+m]}.
$$
Then $\mathcal P_{\mathcal Z}$ is c.c.c.
\end{lemma}
\begin{proof}
Before starting the proof, notice that by using the Pressing Down Lemma
and shrinking $S$ if necessary
we may assume that there is a fixed tree $R_0\in (\mathscr F_n)^{\perp}$ witnessing
that $\sigma_\alpha$ is in $K^{[n]}_\varphi$, for all $\alpha \in S$.
By shrinking $S$ and $Z$ if necessary we may moreover assume that
for every $\alpha\in S$, $\gamma \in Z$, every $i<n$ and $j<m$,
$\sigma_\alpha(i)$ and $\tau_\gamma(j)$ are incomparable in $T$.

Now, assume $\mathcal A$ is an uncountable subset of $\mathcal P_{\mathcal Z}$.
We need to find distinct $p$ and $q$ in $\mathcal A$ which are compatible,
i.e. such that $\meet(p\cup q)\cap K^{[m]}=\emptyset$.
By a standard $\Delta$-system argument we can assume that
all elements of $\mathcal A$ have a fixed size $k$ and are mutually disjoint.
For each $\alpha \in S$ we pick an element $p_\alpha$ of $\mathcal A$
such that $\Ht(\tau)> \alpha$, for all $\tau \in p_\alpha$.

Fix for a moment one such $\alpha$. Since by our assumption
$\sigma_\alpha \cup \tau \restriction \alpha \in L_\varphi^{[n+m]}$,
for all $\tau \in p_\alpha$, we can fix an ordinal $\eta_\alpha <\alpha$
such that for every $\xi \in (\eta_\alpha,\alpha)\cap C$ and every
$\tau \in p_\alpha$ there is a tree $R\in \mathscr F_{n+m}\cap N^{n+m}_\xi$
such that $\sigma_\alpha \restriction \xi \cup \tau \restriction \xi \in R$.
By applying the Pressing Down Lemma and shrinking $S$ again we
may assume that all the ordinals $\eta_\alpha$ are equal to some $\eta^*$.

Now, for each $\alpha \in S$, fix an enumeration
$\{ \upsilon^0_\alpha, \ldots, \upsilon^{l_\alpha-1}_\alpha\}$ of
distinct elements of $\{ \tau \restriction \alpha : \tau \in p_\alpha\}$.
We may assume that there is a fixed integer $l$ such that
$l_\alpha=l$, for all $\alpha \in S$. Moreover, by shrinking $S$
further, we may assume that if $\alpha,\beta\in S$ are distinct
then $\upsilon^i_\alpha$ and $\upsilon^j_\beta$ are incomparable,
for all $i,j<l$.
For each $\alpha \in S$ let
$$
F_\alpha = \{ \sigma_\alpha(j): j< n\}\cup  \{ \upsilon^i_\alpha(j): i<l, j<m\}
$$
and let
$$
D_\alpha = \bigcup \{ D(s,t): s,t\in F_\alpha\}.
$$
Then $D_\alpha$ is finite, so if $\alpha$ is a limit ordinal
and we let $\xi_\alpha = \max (D_\alpha) +1$
then $\xi_\alpha <\alpha$.
By the Pressing Down Lemma and shrinking $S$ yet again
we may assume that there exists
a fixed ordinal $\xi$, a sequence $\sigma\in T_\xi^{[n]}$, and
sequences $\upsilon^i\in T_\xi^{[m]}$, for $i<l$,
such that, for each $\alpha \in S$, we have:
\begin{enumerate}
\item $\xi_\alpha =\xi$,
\item $\sigma_\alpha\restriction \xi =\sigma$,
\item $\upsilon^i_\alpha\restriction \xi=\upsilon^i$, for $i<l$.
\end{enumerate}
Now, notice that if $\alpha,\beta \in S$ are distinct,
then for every $i<l$,
$$
\Delta(\upsilon^i_\alpha,\upsilon^i_\beta) =
\Delta(\sigma_\alpha,\sigma_\beta).
$$
Moreover, if $i$ and $j$ are distinct then

$$
\upsilon^i_\alpha \meet \upsilon^j_\beta =
\upsilon^i_\alpha \meet \upsilon^j_\alpha =
\upsilon^i_\beta \meet \upsilon^j_\beta
\notin  K^{[m]}.
$$
Therefore, $p_\alpha$ and $p_\beta$ are compatible in
$\mathcal P_{{\mathcal Z}}$ provided $\upsilon^i_\alpha \meet \upsilon^i_\beta \notin K^{[m]}$,
for all $i<l$.
We have finally set the stage for the proof of the lemma.

Fix a sufficiently large regular cardinal $\theta$ and a countable
elementary submodel $M$ of $H(\theta)$ containing all the relevant
objects and such that $\delta=M\cap \omega_1$ belongs to $S$.
Working in $M$ fix a countable elementary submodel $N$
of $H(\omega_2)$ containing all the relevant objects
and let $\zeta =N\cap \omega_1$. Now, by our assumption
$\eta^*,C \in N$, $C$ is a club, and $N\in M$, and so
we have that $\zeta \in (\eta^*,\delta)\cap C$.
Therefore, for each $i<l$ there exists
a tree $A_i\in \mathscr F_{n+m}\cap N^{n+m}_\zeta$ such that
$\sigma_\delta \restriction \zeta \cup \upsilon_\delta^i\restriction \zeta \in A_i$.
Since $N^{n+m}_\zeta \subseteq N$ we know that $A_i\in N$, for all $i$.
Let
$$
H=\{ \eta : \exists \alpha \in S [\alpha >\eta \meet
\forall i<l(\sigma_\alpha \restriction \eta \cup \upsilon_\alpha^i\restriction \eta
\in A_i)]\}
$$
Since all the parameters in the definition of $H$ are in $N$, by elementarity of $N$
it follows that $H\in N$. On the other hand $\zeta \in H\setminus N$,
therefore $H$ is uncountable.
Fix a $1-1$ function $f:H\rightarrow S$ with $f\in N$ such that
for every $\eta \in H$, $f(\eta)$ witnesses that $\eta \in H$.
Then the set
$X=\{ \sigma_{f(\eta)}: \eta \in H\}$ belongs to $N$.
We also know that $X$ is an uncountable subset of $R_0$.
Since $T^{[n]}$ is a special tree, by shrinking $H$ we may assume that
$Y=\{ \sigma_{f(\eta)}\restriction \eta: \eta \in H\}$
is an antichain in $T^{[n]}$. Since  $Y\subseteq R_0$
and $R_0\in (\mathscr F_n)^{\perp}$, by Fact~\ref{uncountableX}, there are
distinct $\eta,\rho \in H$ such that:
$$
\sigma_{f(\eta)}\meet \sigma_{f(\rho)}=
\sigma_{f(\eta)}\restriction \eta \meet
\sigma_{f(\rho)}\restriction \rho \in K^{[n]}.
$$
Let $\alpha=f(\eta)$ and $\beta=f(\rho)$.
We claim that $p_{\alpha}$ and $p_{\beta}$
are compatible in $\mathcal P_{{\mathcal Z}}$.
To see this, consider some $i<l$. We know that
$\sigma_{\alpha}\restriction \eta \cup
\upsilon^i_{\alpha} \restriction \eta$ and
$\sigma_{\beta}\restriction \rho \cup
\upsilon^i_{\beta} \restriction \rho$ belong to $A_i$.
Therefore,
$$
(\sigma_\alpha \cup \upsilon^i_\alpha)\meet
(\sigma_\beta \cup \upsilon^i_\beta) =
(\sigma_{\alpha}\restriction \eta \cup
\upsilon^i_{\alpha} \restriction \eta)\meet
(\sigma_{\beta}\restriction \rho \cup
\upsilon^i_{\beta} \restriction \rho) \notin K^{[n+m]}.
$$
Since
$\sigma_{\alpha} \meet
\sigma_{\beta}\in K^{[n]}$
it follows that
$$
\upsilon^i_{\alpha} \meet
\upsilon^i_{\beta}  \notin K^{[m]}.
$$
Since this is true for all $i$ it follows that
$p_{\alpha}$ and $p_{\beta}$ are compatible.
\end{proof}

\begin{lemma} [$\MA_{\aleph_1}$]\label{ma_lemma}
Let ${\mathcal S}=\{ \sigma_\alpha: \alpha \in S\}$ and
${\mathcal Z} =\{ \tau_\gamma: \gamma \in Z\}$ be
two regular level sequences in $T^{[n]} $ and $T^{[m]}$
respectively such that $S$  and $Z$ are stationary subsets
of $C$. Assume ${\mathcal S} \subseteq K^{[n]}_\varphi$ and
${\mathcal Z} \subseteq K^{[m]}_\varphi$. Then there exist
$\alpha \in S$ and $\gamma \in Z$ such that
$\alpha <\gamma$ and
$$
\sigma_\alpha \cup \tau_\gamma \restriction \alpha \in K_\varphi^{[n+m]}.
$$
\end{lemma}

\begin{proof}
For every $\gamma \in Z$ fix a tree
$R_\gamma \in (\mathscr F_m)^{\perp}\cap N_\gamma^m$
such that $\tau_\gamma \in R_\gamma$, i.e.
witnessing that $\tau_\gamma \in K^{[m]}_{\varphi}$.
Since $Z$ is stationary by the Pressing Down Lemma
and shrinking $Z$ if necessary we may assume that
all the $R_\gamma$ are equal to some tree $R$.
Assume towards contradiction that for every $\alpha \in S$
and $\gamma \in Z$, if $\alpha <\gamma$ then
$
\sigma_\alpha \cup \tau_\gamma \restriction \alpha \in L_\varphi^{[n+m]}.
$
By Lemma~\ref{mainlemma} $\mathcal P_{\mathcal Z}$ is c.c.c.
By $\MA_{\aleph_1}$ we can find an uncountable subset
$Y$ of $Z$ such that
$\tau_\alpha \meet \tau_\beta \notin K^{[m]}$,
for every distinct $\alpha, \beta \in Y$.
This means that the tree $R^*$ generated by $\{ \tau_\alpha: \alpha \in Y\}$
belongs to $\mathscr F_m$. However, $R^*\subseteq R$
and $R$ is orthogonal to all trees in $\mathscr F_m$, a contradiction.
\end{proof}

\section{The forcing $\partial^*(K)$}

In this section we define a notion of forcing  $\partial^*(K)$ and
prove that it is proper. We then show that either there
is an uncountable antichain $Y$ in $T$ such that $\meet(Y)\cap K=\emptyset$
or forcing with $\partial^*(K)$ adds an uncountable subset $X$ of $K_\varphi \cap T\restriction C$
such that $\pi(X)=X$.
Then it will be easy to force again and obtain an uncountable subset
$Z$ of $X$ such that $\meet(Z)\subseteq K$.
Before we start it will be convenient to define a certain club of
countable elementary submodels of $H(\omega_2)$.
Fix, for each $\delta <\omega_1$,  a surjection
$e_\delta:\omega \rightarrow T_\delta$.

\begin{defi} $\mathcal E$ is the collection of all
countable elementary submodels $M$ of $H(\omega_2)$ such
that  $T,C,K,  (e_\delta : \delta <\omega_1)$
as well as $(N^n_\xi: \xi \in C)$, for $n<\omega$,
all belong to $M$.
\end{defi}

We are now in the position to define the partial order
$\partial^*(K)$.

\begin{defi}
 $\partial^*(K)$ consists of all pairs $(X_p,\mathcal M_p)$
such that:
\begin{enumerate}
\item $X_p$ is a finite subset of $T\restriction C$, $\pi(X_p)=X_p$, and
$X_p\cap T_\alpha \in K_\varphi$,
for all $\alpha \in \lev(X_p)$\footnote{Here, of course, we identify
$X_p\cap T_\alpha$ with its $\leq_{\rm lex}$-increasing enumeration.}.

\item $\mathcal M_p$ is a finite $\in$-chain of elements of $\mathcal E$
such that for every $x\in X_p$ there is $M\in \mathcal M_p$ such
that $\Ht(x)=M\cap \omega_1$.
\end{enumerate}
The order of $\partial^*(K)$ is the coordinatewise reverse inclusion,
i.e. $q\leq p$ iff $X_p\subseteq X_q$ and $\mathcal M_p\subseteq \mathcal M_q$.
\end{defi}

In what follows, for $p \in \partial^*(K)$, $M_{p}^i$ denotes the $i$-th model
in $\mathcal{M}_p$, in the enumeration induced by the heights of the models.

\begin{teo}\label{proper}
$\partial^*(K)$ is a proper forcing notion.
\end{teo}
\begin{proof}
Fix a countable $M \prec H(2^{|\partial^*(K)|^+})$ such that
$\partial^*(K),{\mathcal E} \in M$. Given a condition $p = (X_p, \mathcal{M}_p) \in M$,
we need to find $q \leq p$ that is $(\partial^*(K), M)$-generic.
Set
$$
q = (X_p, \mathcal{M}_p \cup \{M \cap H(\omega_2)\}).
$$

We claim that $q$ is as desired. To see this, fix a dense set $D \in M$
and a condition $r \leq q$. We need to find $s\in D\cap M$ which is
compatible with $r$. By replacing $r$ with a stronger condition
we may assume that $r\in D$.
Define
$$
r' = (X_r \cap M, \mathcal{M}_r \cap M)
$$
and
$$
r^* = (X_r \setminus M, \mathcal{M}_r \setminus M)
$$
Note that $r'$ and $r^*$ are both conditions in $\partial^*(K)$.
Let $l=|\mathcal{M}_r \setminus M|$.
For every $i < l$, let $\delta_i = M_{r^*}^{i} \cap \omega_1$,
let $k_i=|X_{r^*}\cap T_{\delta_i}|$ and fix $\sigma_i\in [\omega]^{k_i}$
such that $e_{\delta_i}[\sigma_i]= X_{r^*}\cap T_{\delta_i}$.
As before, we also think of $X_{r^*}\cap T_{\delta_i}$ as an element
of $T^{[k_i]}$ via its $\leq_{\rm lex}$-increasing enumeration.
We now define formulas $\theta_i$, for $i<l$, by reverse
induction on $i$.

$\theta_{l}(\xi_0,\ldots,\xi_{l-1})$ holds if there is
 a condition $s=(X_s,\mathcal M_s)\in \partial^*(K)$ such that:
\begin{enumerate}
\item $\mathcal M_s =\{ M_s^0,\ldots,M_s^{l-1}\}$,
\item $M_s^i\cap \omega_1=\xi_i$, for all $i<l$,
\item $X_s\cap T_{\xi_i} = e_{\xi_i}[\sigma_i]$ and $|X_s\cap T_{\xi_i}|=k_i$, for all $i<l$,
\item $(X_{r'} \cup X_s, \mathcal{M}_{r'} \cup \mathcal{M}_s) \in D$.
\end{enumerate}

Suppose $\theta_{i+1}$ has been defined for some $i< l$. Then
$$
\theta_{i}(\xi_0,\ldots,\xi_{i-1})
\mbox{ iff } Q \eta \, \theta_{i+1}(\xi_0,\ldots,\xi_{i-1},\eta).
$$
Here $Q \eta \, \theta(\eta)$ means ``there are stationary many $\eta$
such that $\theta(\eta)$ holds".

\begin{remark}
\label{absolute}
 Notice that the parameters of each $\theta_i (\xi_{0}, \ldots, \xi_{i-1})$
are in $M$, so if $\xi_0,\ldots,\xi_{i-1} \in M$
then $\theta_i(\xi_0,\ldots,\xi_{i-1})$ holds iff it holds
in $M$. Thus, if $W_i$ is the set of tuples $(\xi_0,\ldots,\xi_{i-1})$
such that $\theta_{i}(\xi_0,\ldots,\xi_{i-1})$ holds then $W_i\in M$,
for all $i \leq l$. We set $W=\bigcup_{i\leq l}W_i$.

 Notice also  that  if $\theta_i(\xi_0,\ldots,\xi_{i-1})$ holds
then $e_{\xi_j}[\sigma_j]\in K^{[k_j]}_\varphi$, for all $j<i$.

\end{remark}

\begin{cla}
\label{induction}
$\theta_i(\delta_0,\ldots,\delta_{i-1})$ holds, for all $i\leq l$.
\end{cla}

\begin{proof} We prove this by reverse induction on $i$.
Notice that $\theta_{l}(\delta_0,\ldots,\delta_{l-1})$
holds as witnessed by the condition $r^*$.
Suppose we have established $\theta_{i+1}(\delta_0,\ldots,\delta_{i})$.
Since $W_{i+1}\in M\cap H(\omega_2)$ and $M \cap H(\omega_2) = M_{r^*}^0 \subseteq M_{r^*}^{i}$
it follows that the set
$$
Z=\{ \eta: (\delta_0,\ldots,\delta_{i-1},\eta)\in W_{i+1}\}
$$
also belongs to $M_{r^*}^{i}$. If $Z$ were non stationary,
by elementarity, there would be a club $E\in M_{r^*}^{i}$ disjoint
from it, but $\delta_i \in Z$ and $\delta_i$ belongs to any
club which is in $M_{r^*}^{i}$, a contradiction.
\end{proof}

By Claim~\ref{induction} we can pick a {\em stationary splitting
tree} $U\subseteq W$. This means that $U\subseteq (\omega_1)^{\leq l}$ is a tree
and for every node $t=(\xi_0,\ldots,\xi_{i-1})\in U $ of height $ i<l$ the set
$$
S_t=\{ \eta: (\xi_0,\ldots,\xi_{i-1},\eta) \in U\}
$$
is stationary. We can moreover assume that $U \in M$.
We now build by induction an increasing sequence $(\xi_i :i<l)$ of
ordinals in $M$ such that:

\medskip

\begin{enumerate}
\item $(\xi_0,\ldots,\xi_i)\in U$, for all $i$,
\item $e_{\xi_i}[\sigma_i]\cup e_{\delta_0}[\sigma_0]\restriction \xi_i \in K_\varphi^{[k_i+k_0]}$, for all $i$.
\end{enumerate}
\medskip
The reason we require $(2)$ is that
we want the condition $s$ witnessing 
the formula $\theta_l(\xi_0,\ldots,\xi_{l-1})$ to be compatible with $r$
and this is the only  problem we may encounter. 
Now, suppose $j<l$ and we have picked $\xi_i$, for all $i<j$. Consider the set
$S_j=\{ \eta : (\xi_0,\ldots,\xi_{j-1},\eta)\in U\}$.

\begin{cla} There is $\xi \in S_j\cap M$ such that
$e_{\xi}[\sigma_j]\cup e_{\delta_0}[\sigma_0]\restriction \xi  \in K^{[k_j+k_0]}_\varphi$.
\end{cla}

\begin{proof} Assume otherwise. We know that
$S_j$ is stationary and that the level sequence
$\mathcal S_j = \{ e_{\eta}[\sigma_j]: \eta \in S_j\}$ is contained
in $K_\varphi^{[k_j]}$. By shrinking $S_j$  we may also assume that
$\mathcal S_j$ is regular.
Let
$$
Z= \{ \eta : e_{\eta}[\sigma_0]\in K^{[k_0]}_\varphi \meet \forall \xi \in S_j\cap \eta
[e_{\xi}[\sigma_j]\cup e_{\eta}[\sigma_0]\restriction \xi \notin K_\varphi^{[k_j+k_0]}]\}.
$$
Then $Z\in M$ and since by our assumption  $\delta_0\in Z$,
it follows that $Z$ is stationary. By shrinking $Z$,
we may assume that the level sequence $\mathcal Z=\{ e_\eta[\sigma_0]: \eta \in Z\}$
is regular. Now, by Lemma~\ref{ma_lemma} and $\MA_{\aleph_1}$ we obtain
a contradiction.
\end{proof}

Suppose $(\xi_0,\ldots, \xi_{l-1})$ has been constructed.
Since $(\xi_0,\ldots,\xi_{l-1})\in U\cap M$, by elementarity there is
a condition $s\in \partial^*(K)\cap M$ witnessing this fact.
Let
$$
\bar{s}=(X_{r'}\cup X_{s}, \mathcal M_{r'}\cup \mathcal M_{s}).
$$
Then by (4) in the statement of $\theta_l(\xi_0,\ldots,\xi_{l-1})$  we know that
$\bar{s}\in D$. Since $s$ and $r'$ are both in $M$ so is
$\bar{s}$. We claim that $\bar{s}$ is compatible with $r$.
To see this we define a condition $u$ as follows.
Let
$$
X_u = \pi( X_r \cup X_{s}).
$$
Note that $\lev(X_u)=\lev(X_{r'})\cup \lev(X_s)\cup \lev(X_{r^*})$ and
we have
\begin{equation*}
X_u\cap T_{\alpha}=
\begin{cases}
e_{\delta_i}[\sigma_i] & \text{if $\alpha =\delta_i$, for some $i<l$,}
\\
e_{\xi_i}[\sigma_i]\cup e_{\delta_0}[\sigma_0]\restriction \xi_i & \text{if $\alpha = \xi_i$, for some $i<l$,}
\\
X_{r'} \cap T_\alpha &  \text{if $\alpha \in \lev(X_{r'})$.}
\end{cases}
\end{equation*}
In all cases we have that $X_u\cap T_\alpha \in K_\varphi$.
We let $\mathcal M_{u}=\mathcal M_{r'}\cup \mathcal M_{s}\cup M_{r^*}$.
It follows that $u\leq \bar{s},r$.
This completes the proof of Theorem~\ref{proper}.

\end{proof}

\section{The main theorem}

In this section we complete the proof of the main theorem saying that
the conjunction of  $\BPFA$ and $\varphi$ implies $\CAT$.
Let $G$ be $V$-generic over the poset $\partial^*(K)$ and define
in $V[G]$:
$$
X_G = \bigcup \{ X_p: p\in G \}.
$$
Note that $\pi(X_G)=X_G$, $\lev(X_G)\subseteq C$, and every finite subset of $X_G$
contained in one level of $T$ is in $K_\varphi$.
Let $\dot{X}_G$ be a canonical $\partial^*(K)$-name for $X_G$.
We first establish the following fact in the ground model $V$.

\begin{lemma}
\label{density}
Suppose there is no uncountable antichain $Y$ in $T$ such that
$\meet(Y)\cap K=\emptyset$. Then there is a condition $p\in \partial^*(K)$
which forces that $\dot{X}_G$ is uncountable.
\end{lemma}

\begin{proof} Suppose the maximal condition forces that $\dot{X}_G$ is countable.
Let $\theta$ be a sufficiently large regular cardinal and let $M$ be a countable
elementary submodel of $H(\theta)$ containing all the relevant objects.
As shown in Theorem~\ref{proper}  $q=(\emptyset, \{ M\cap H(\omega_2)\})$ is an
$(M,\partial^*(K))$-generic condition. Therefore, $q\forces \dot{X_G}\subseteq M$.
If there is a node $t$ in $K^{[1]}_\varphi \cap T_\delta^{[1]}$  then
$r=(\{ t\}, \{ M\cap H(\omega_2)\})$ is a condition stronger than $q$
and $r\forces t\in \dot{X}_G$, a contradiction.
Assume now that $T^{[1]}_\delta \subseteq L^{[1]}_\varphi$. Since $\delta$
is a limit point of $C$, by Fact~\ref{countableB} there is an uncountable
antichain $Y$ in $T$ such that $\meet(Y)\cap K =\emptyset$, again a contradiction.
\end{proof}

Now, assume there is no uncountable antichain $Y\subseteq T$
such that $\meet(Y)\cap K=\emptyset$
and fix a $V$-generic $G$ over $\partial^*(K)$ containing a condition
as in Lemma~\ref{density}. We work in $V[G]$.
We can show that $\lev(X_G)$ is a club, but this is not necessary.
Namely, let $\bar{X}_G$ be the closure of $X_G$ in the tree topology.
Then by Fact~\ref{cont} all finite subsets of $\bar{X}_G$ contained
in one level of $T$ are in $K_\varphi$. Moreover, $\lev(\bar{X}_G)$ is
equal to the closure of $\lev(X_G)$ in the order topology and is a club.
Clearly, we also have $\pi(\bar{X}_G)=\bar{X}_G$.

\begin{remark}
Before continuing it is important to note a certain amount of absoluteness
between $V$ and $V[G]$. In $V$ we defined $(\mathscr F_n)^V$ to be the collection
of regular subtrees $R$ of $T^{[n]}$ such that $\meet(R)\cap K^{[n]}=\emptyset$.
The same definition in $V[G]$ gives a larger collection $\mathscr F_n^{V[G]}$
of subtrees of $T^{[n]}$. Nevertheless, the definition of $\mathscr F_n$
is $\Sigma_1$ with parameters $T$ and $K$. Since $\BPFA$ holds in $V$,
if a certain tree $A\in V$ is in $(\mathscr F_n^V)^{\perp}$ then
it is also in $(\mathscr F_n^{V[G]})^{\perp}$.
It follows that the same sequences $(N^n_\xi : \xi \in C)$
witness $\varphi(\mathscr F_n)$ in $V$ and in $V[G]$, for all $n$.
Therefore the definitions of the induced colorings $K^{[n]}_\varphi$, for
$n<\omega$, are also absolute between $V$ and $V[G]$.
\end{remark}

\begin{defi}
The poset $\mathcal Q$ consists of finite antichains $p$ in $\bar{X}_G$
such that $\meet(p)\subseteq K$, ordered by reverse inclusion.
\end{defi}

\begin{cla} $\mathcal Q$ is a c.c.c. poset.
\end{cla}

\begin{proof}
Suppose $\mathcal A$ is an uncountable subset of $\mathcal Q$.
We need to find two elements of $A$ which are compatible.
By a standard $\Delta$-system argument we may assume
that the elements of $\mathcal A$ are disjoint and have
the same size. For each $\alpha \in \lev(\bar{X}_G)$
choose $p_\alpha \in \mathcal A$ such that $\Ht(t)\geq \alpha$,
for all $t \in p$. We can assume that the $p_\alpha$ are
distinct. Let $\sigma_\alpha$ be the enumeration
in $\leq_{\rm lex}$-increasing order of distinct elements of
$\{ t\restriction \alpha : t\in p_\alpha\}$.
There is a stationary subset $S$ of $\lev(\bar{X}_G)$
and an integer $n$ such that $\sigma_\alpha$ has size
$n$, for all $\alpha\in S$. Note that
$\sigma_\alpha \in K^{[n]}_\varphi$, for all $\alpha \in S$.
By shrinking $S$ further we may assume that $\{ \sigma_\alpha : \alpha \in S\}$
is a regular level sequence and that for every $\alpha, \beta \in S$
and every distinct $i,j<n$
$$
\sigma_\alpha(i)\meet \sigma_\beta(j) =
\sigma_\alpha(i)\meet \sigma_\alpha(j)\in K.
$$
Now, by Fact~\ref{gen} we can find distinct $\alpha,\beta \in S$
such that $\sigma_\alpha \meet \sigma_\beta \in K^{[n]}$,
i.e. $\sigma_\alpha(i)\meet \sigma_\beta(i)\in K$, for all
$i<n$. It follows that $p_\alpha$ and $p_\beta$ are compatible
in $\mathcal Q$.
\end{proof}

By a standard argument there is a condition $q\in \mathcal Q$
which forces the $\mathcal Q$-generic $H$ to be uncountable.
Therefore, by forcing with $\mathcal Q$ below $q$ over $V[G]$
we obtain an uncountable antichain $H$ of $T$ such that
$\meet(H)\subseteq K$. Since $\partial^*(K)* \mathcal Q$
is proper, by $\BPFA$, we have such an antichain in $V$.
Thus, we have proved the main theorem which we now state.

\begin{teo}
Assume $\BPFA$ and $\varphi$. Then $\CAT$ holds and hence
there is a five element basis for the uncountable linear orders.
\qed
\end{teo}

\def\Dbar{\leavevmode\lower.6ex\hbox to 0pt{\hskip-.23ex \accent"16\hss}D}

\end{document}